\documentclass{elsarticle}%
\usepackage{amsmath}%
\usepackage{amsfonts}%
\usepackage{amssymb}%
\usepackage{amsthm}
\usepackage{mathrsfs}
\usepackage{array}
\usepackage{tikz}
\usetikzlibrary{fit}
\usepackage{adjustbox}
\usepackage[left=1.5in, right=1.5in]{geometry}

\newcommand{\Q}{\mathbb{Q}}
\newcommand{\C}{\mathbb{C}}
\newcommand{\Z}{\mathbb{Z}}

\newcommand{\mf}{\mathfrak}

\newcommand{\Syl}{\text{Syl}}

\renewcommand{\ker}{\text{\normalfont{ker}}}
\newcommand{\coker}{\text{\normalfont{coker}}}

\newcommand{\diag}{\text{\normalfont{diag}}}

\renewcommand{\bar}{\overline}
\renewcommand{\tilde}{\widetilde}

\newtheorem{theorem}{Theorem}
	
\newtheorem{prop}[theorem]{Proposition}
\newtheorem{cor}[theorem]{Corollary}
\newtheorem*{thm3}{Theorem \ref{Lorenzini applied}}
\newtheorem*{thm5}{Theorem \ref{reflection rep}}
\theoremstyle{definition}
\newtheorem{definition}[theorem]{Definition}
\newtheorem{example}[theorem]{Example}

\theoremstyle{remark}
\newtheorem*{remark}{Remark}

\begin{document}
\title{Critical groups of group representations}
\author[umn]{Christian Gaetz \fnref{fn1}}
\ead{gaetz@mit.edu}
\fntext[fn1]{Present address: Department of Mathematics, Massachusetts Institute of Technology, Cambridge, MA 02139.}

\address[umn]{School of Mathematics, University of Minnesota, Minneapolis, MN 55455.}

\date{\today}

\begin{abstract}
This paper investigates the critical group of a faithful representation of a finite group.  It computes the order of the critical group in terms of the character values, and gives some restrictions on its subgroup structure.  It also computes the precise structure of the critical group both for the regular representation of any finite group, and for the reflection representation of $\mf{S}_n$. 
\end{abstract}

\maketitle


\section{Introduction} \label{intro}  
Benkart, Klivans, and Reiner recently defined a new invariant of a faithful representation $\gamma$ of a finite group called its \textit{critical group}, denoted $K(\gamma)$, a finite abelian group \cite{BKR}.  It was motivated by, and in some cases coincides with, the critical group of a directed graph.  For a brief introduction to the directed graph case see \cite{LP}, where critical groups are called \textit{sandpile groups}, and for an example of when these ideas coincide see Example \ref{coincide} or \cite[Section 6.3]{BKR}.

Let $G$ be a finite group with irreducible complex characters $\chi_0,...,\chi_{\ell}$.  We will always let $\chi_0$ denote the character of the trivial representation.  Let $\gamma$ be a faithful (not necessarily irreducible) $n$-dimensional representation of $G$ with character $\chi_{\gamma}$.  Let $M \in \Z^{(\ell +1) \times (\ell+1)}$ be the matrix whose entries are defined by the equations
\[
\chi_{\gamma} \cdot \chi_i = \sum_{j=0}^{\ell} m_{ij} \chi_j
\]
\begin{definition}
Let $\gamma$ be an $n$-dimensional faithful complex representation of a finite group $G$, and let $M$ be defined as above.  The \textit{extended McKay-Cartan matrix} is $\tilde{C}:=nI-M$ and the \textit{McKay-Cartan matrix} $C$ is the submatrix of $\tilde{C}$ obtained by removing the row and column corresponding to the trivial character $\chi_0$.
\end{definition}

\begin{definition}\cite[Definition 5.11]{BKR}
Given a faithful complex representation $\gamma$ of a finite group $G$ with McKay-Cartan and extended McKay-Cartan matrices $C$ and $\tilde{C}$, we define its \textit{critical group} in either of the following equivalent ways:
\begin{align*}
K(\gamma)&:= \coker(C^t)=K(C) \\
\Z \oplus K(\gamma) &:= \coker(\tilde{C}^t)
\end{align*}
\end{definition}
 
Section \ref{structure} applies some results of Lorenzini in \cite{Lor} to obtain the first main Theorem:
\begin{theorem} \label{Lorenzini applied}
Let $G$ be a finite group with faithful complex $n$-dimensional representation $\gamma$ and critical group $K(\gamma)$.  Let $e=c_0,c_1,...,c_{\ell}$ be a set of conjugacy class representatives for $G$, then:
\begin{enumerate}
\item[(i)] 
\[
\frac{1}{|G|} \prod_{i=1}^{\ell} (n-\chi_{\gamma}(c_i))=|K(\gamma)|
\]

\item[(ii)]  If $\chi_{\gamma}$ is real-valued, and $\chi_{\gamma}(c)$ is an integer character value achieved by $m$ different conjugacy classes, then $K(\gamma)$ contains a subgroup isomorphic to \\ $(\Z/(n-\chi_{\gamma}(c)) \Z)^{m-1}$.
\end{enumerate}
\end{theorem}
\begin{cor} \label{regular cor}
The regular representation $\gamma$ of a finite group $G$ has $K(\gamma) \cong (\Z/|G| \Z)^{\ell-1}$.
\end{cor}

Finally, Section \ref{symmetric group} derives an explicit formula for the critical group in the case of the reflection representation of $\mf{S}_n$:

\begin{theorem}
\label{reflection rep}
Let $\gamma$ be the reflection representation of $\mf{S}_n$ and let $p(j)$ denote the number of partitions of the integer $j$.  Then
\[
K(\gamma) \cong \bigoplus_{i=2}^{p(n)-p(n-1)} \Z/q_i \Z
\]
where
\[
q_i=\prod_{\substack{1 \leq j \leq n \\ p(j)-p(j-1)\geq i}}j
\]
\end{theorem}

\begin{example}
Let $G=\mf{S}_4$ and consider the reflection representation $\gamma$ of $G$ (this is the action by permutation matrices on $\C^4$ with the copy of the trivial representation removed; the associated integer partition is $(3,1)$).  The character table is
\begin{center}
\begin{tabular}{|c|c|c|c|c|c|}
\hline
& $e$ & $(12)$ & $(123)$ & $(1234)$ & $(12)(34)$ \\ \hline
$\chi_0$ & 1 & 1& 1 & 1 & 1\\ \hline
$\chi_{\gamma}=\chi_1$ & 3 & 1 & 0 & -1 & -1 \\ \hline
$\chi_2$ & 2 & 0 & -1 & 0 & 2 \\ \hline
$\chi_3$ & 3 & -1 & 0 & 1 & -1 \\ \hline
$\chi_4$ & 1 & -1 & 1 & -1 & 1 \\ \hline
\end{tabular}
\end{center}
The matrices $M, \tilde{C}$ and $C$ associated to $\gamma$ are
\begin{align*}
&M = \begin{pmatrix}
0 & 1 & 0 & 0 & 0 \\ 1 & 1 & 1 & 1 & 0 \\ 0 & 1 & 0 & 1 & 0 \\ 0 & 1 & 1 & 1 & 1 \\ 0 & 0 & 0 & 1 & 0
\end{pmatrix}
&\tilde{C} = \begin{pmatrix}
3 & -1 & 0 & 0 & 0 \\ -1 & 2 & -1 & -1 & 0 \\ 0 & -1 & 3 & -1 & 0 \\ 0 & -1 & -1 & 2 & -1 \\ 0 & 0 & 0 &-1 & 3
\end{pmatrix}
&\hspace{.15in} C = \begin{pmatrix}
2 & -1 & -1 & 0 \\ -1 & 3 & -1 & 0 \\ -1 & -1 & 2 & -1 \\ 0 & 0 & -1 & 3
\end{pmatrix}
\end{align*}
To calculate $K(\gamma)$, we calculate that $C$ has Smith normal form $\diag(1,1,1,4)$, or equivalently that $\tilde{C}$ has Smith form $\diag(1,1,1,4,0)$.  Thus $K(\gamma) \cong \Z/4 \Z$.  

This example illustrates Theorems \ref{Lorenzini applied} and \ref{reflection rep}.  By Theorem \ref{Lorenzini applied} (i), we know 
$$|K(\gamma)|=\frac{1}{4!} (3-1)(3-0)(3+1)(3+1)=4$$
and by part (ii) we know that $K(\gamma)$ has a subgroup isomorphic to $\Z/4 \Z$, since $\chi_{\gamma}$ has a repeated integer character value -1.  Since $\gamma$ is the reflection representation of the symmetric group $\mf{S}_4$, Theorem \ref{symmetric group} also applies.  The relevant partition numbers are $p(1)=1, p(2)=2, p(3)=3$, and $p(4)=5$.  Therefore $q_2=4$ and Theorem \ref{reflection rep} then implies $K(\gamma) \cong \Z/4\Z$. 
\end{example}

\section{Proof of Theorem \ref{Lorenzini applied}} \label{structure}

Recall that for $\Gamma$ a simple connected graph with Laplacian matrix $\tilde{L}$ its \textit{critical group} is defined as $K(\Gamma):= \coker(L)$ where $L$ is obtained from $\tilde{L}$ by removing any row and corresponding column.  If $\Gamma$ has $\ell+1$ vertices and $\tilde{L}$ has eigenvalues $0=\lambda_0, \lambda_1,..., \lambda_{\ell}$ then a classical result (see for example \cite[Corollary 6.5]{Biggs} and Kirchhoff's Matrix Tree Theorem) expresses the size of the critical group as:
\[
|K(\Gamma)|=\frac{\lambda_1 \cdots \lambda_{\ell}}{\ell + 1}
\]
Theorem \ref{Lorenzini applied} (i) gives the analogous result for critical groups of group representations.  The following general result of Lorenzini about $(\ell+1) \times (\ell+1)$-integer matrices of rank $\ell$ will be useful.

\begin{prop}\cite[Propositions 2.1 and 2.3]{Lor}
\label{Lorenzini prop}
Let $M$ be any $(\ell+1) \times (\ell+1)$-integer matrix of rank $\ell$ with characteristic polynomial $\normalfont{char}_M(x)=x \prod_{i=1}^{\ell} (x-\lambda_i)$.  Let $R$ be an integer vector in lowest terms generating the kernel of $M$, and let $R'$ be the corresponding vector for $M^t$.  Let $H$ be the torsion subgroup of the cokernel of $\Z^{\ell+1} \xrightarrow{M} \Z^{\ell+1}$.
\begin{enumerate}
\item[(i)] 
\[
\prod_{i=1}^{\ell}\lambda_i= \pm |H| (R \cdot R')
\]
where $R \cdot R'$ denotes the dot product.
\item[(ii)] Let $\lambda \neq \pm 1, 0$ be an integer eigenvalue of $M$, and $\mu(\lambda)=\dim_{\Q} \ker(M-\lambda I)$.  If $M$ is symmetric and $R$ has at least one entry with value $\pm 1$ then $C$ contains a subgroup isomorphic to $(\Z/\lambda \Z)^{\mu(\lambda)-1}$.
\end{enumerate}
\end{prop}

Translating this proposition into the context of critical groups of group representations allows us to prove Theorem \ref{Lorenzini applied}.

\begin{prop}\cite[Proposition 5.3]{BKR}
\label{eigenvectors}
Let $\gamma$ be a faithful complex representation of a finite group $G$.
\begin{enumerate}
\item[(i)] A full set of orthogonal eigenvectors for $\tilde{C}$ is given by the set of columns of the character table of $G$:
\[
\delta^{(g)}=(\chi_0(g),...,\chi_{\ell}(g))^t
\]
as $g$ ranges over a collection of conjugacy class representatives for $G$.
\item[(ii)] The corresponding eigenvalues are given by
\[
\tilde{C} \delta^{(g)}=(n-\chi_{\gamma}(g))\cdot \delta^{(g)}
\]
\item[(iii)] The vector $\delta^{(e)}=(\chi_0(e),...,\chi_{\ell}(e))^t$ spans the nullspace of both $\tilde{C}$ and $\tilde{C}^t$.  In particular, this implies that these matrices have rank $\ell$.
\end{enumerate}
\end{prop}

We recall the statement of Theorem \ref{Lorenzini applied}.
\begin{thm3}
Let $G$ be a finite group with faithful complex $n$-dimensional representation $\gamma$ and critical group $K(\gamma)$.  Let $e=c_0,c_1,...,c_{\ell}$ be a set of conjugacy class representatives for $G$, then:
\begin{enumerate}
\item[(i)] 
\[
\frac{1}{|G|} \prod_{i=1}^{\ell} (n-\chi_{\gamma}(c_i))=|K(\gamma)|
\]

\item[(ii)]  If $\chi_{\gamma}$ is real-valued, and $\chi_{\gamma}(c)$ is an integer character value achieved by $m$ different conjugacy classes, then $K(\gamma)$ contains a subgroup isomorphic to \\ $(\Z/(n-\chi_{\gamma}(c)) \Z)^{m-1}$.
\end{enumerate}
\end{thm3}
\begin{proof}
Part (i) follows from Proposition \ref{Lorenzini prop} (i) with $M=\tilde{C}$: the eigenvalues of $\tilde{C}$ are given in Proposition \ref{eigenvectors} (ii), and we have that in this case $R=R'=(\chi_0(e),...,\chi_{\ell}(e))^t$ by Proposition \ref{eigenvectors} (iii).  Both $R$ and $R'$ are in lowest terms since $\chi_0(e)=1$.  Now, we have
\[
R \cdot R' = \sum_{i=0}^{\ell} \dim(\chi_i)^2 = |G|
\] 
To see that the left-hand-side in (i) is always positive, notice that when $\chi_{\gamma}(c_i)$ is real, then $n-\chi_{\gamma}(c_i)>0$.  If $\chi_{\gamma}(c_i)$ is complex, then $\bar{\chi_{\gamma}(c_i)}=\chi_{\gamma}(c_j)$ where $c_j$ is conjugate to $c_i^{-1}$.  In this case
\[
(n-\chi_{\gamma}(c_i))(n-\chi_{\gamma}(c_j))=|n-\chi_{\gamma}(c_i)|^2 >0 
\]

For part (ii), notice that $\mu(n-\chi_{\gamma}(c))=m$ since the eigenvectors of $\tilde{C}$ are orthogonal by Proposition \ref{eigenvectors} (i).  If $\chi_{\gamma}$ is real valued then
\begin{align*}
&\langle \chi_j, \chi_{\gamma} \chi_i \rangle = \frac{1}{|G|} \sum_{g \in G} \chi_j(g) \bar{\chi_{\gamma}(g) \chi_i(g)} \\
&= \frac{1}{|G|} \sum_{g \in G} \chi_j(g) \chi_{\gamma}(g) \bar{ \chi_i(g)} = \langle \chi_j \chi_{\gamma}, \chi_i \rangle
\end{align*}
so $\tilde{C}$ is symmetric in this case.  Finally, $\chi_{\gamma}(c) < n$, so we can rule out the cases $\lambda=0,-1$ in Proposition \ref{Lorenzini prop} (ii).  Applying this proposition then gives the desired result.
\end{proof}

The following interesting application of Theorem \ref{Lorenzini applied} was suggested by Vic Reiner:

\begin{example} \label{coincide}
Let $G$ be any finite group, and let $\gamma$ be the regular representation of $G$, which is always faithful.  We know that $\dim(\gamma)=|G|$ and that $\chi_{\gamma}$ is 0 on the non-identity conjugacy classes.  Therefore Theorem \ref{Lorenzini applied} (i) gives that $|G|^{\ell-1}=|K(\gamma)|$.  Furthermore, $\chi_{\gamma}$ is real-valued and the value 0 is achieved on $\ell$ conjugacy classes, so part (ii) implies that $K(\gamma)$ has a subgroup isomorphic to $(\Z/|G|\Z)^{\ell-1}$.  Combining these results shows that in fact $K(\gamma) \cong (\Z/|G|\Z)^{\ell-1}$, proving Corollary \ref{regular cor}.  In the special case that $G$ is abelian, this shows that $K(\gamma)\cong (\Z/|G|\Z)^{|G|-2}$.  Finally, if $G \cong \Z/n\Z$ then $\tilde{C}$ is equal to the graph Laplacian of the complete graph $K_n$.  This formula then recovers the well-known result that the critical group of the complete graph is $(\Z/n\Z)^{n-2}$.
\end{example}

\begin{cor} \label{Syl}
In the context of Theorem \ref{Lorenzini applied}, if $\chi_{\gamma}$ is $\Q$-valued, then $\Syl_p(K(\gamma))=0$ unless $p \leq 2n$.
\end{cor}
\begin{proof}
The fact that $\chi_{\gamma}$ is $\Q$-valued implies that it is $\Z$-valued since a sum of roots of unity is rational if and only if it is integral.  Thus $-n \leq \chi_{\gamma}(c_i) < n$ for all $i$, and so the product on the left hand side of part (i) of the theorem is a product of integers which are at most $2n$. 
\end{proof}

\begin{example}
The requirement that $\chi_{\gamma}$ is $\Q$-valued in Corollary \ref{Syl} is necessary.  Indeed, if $G=\langle g | g^m=1 \rangle$ is a cyclic group and $G \xrightarrow{\gamma} GL_2(\C)$ is given by
\[
g \mapsto \begin{pmatrix}
e^{2 \pi i/m} & 0 \\ 0 & e^{- 2 \pi i / m}
\end{pmatrix}
\]
then an easy calculation shows that $K(\gamma) \cong \Z/m \Z$.  In this case $|K(\gamma)|$ may have prime divisors up to $m$ which are larger than $2n=4$.
\end{example}

The following easy corollary of Theorem \ref{Lorenzini applied} was first observed by H. Blichfeldt in \cite[Corollary 2]{B}.

\begin{cor}
For any $n$-dimensional representation $\gamma$ of $G$ 
\[
\frac{1}{|G|} \prod_{i=1}^{\ell} (n-\chi_{\gamma}(c_i))
\]
lies in $\Z_{ \geq 0}$.
\end{cor}
\begin{proof}
If $\gamma$ is faithful, Theorem \ref{Lorenzini applied} (i) implies that this expression is the order of $K(\gamma)$.  Otherwise $\chi_{\gamma}(c_i)=n$ for some $c_i \neq e$, so the product is 0 in this case.
\end{proof}

\section{Proof of Theorem \ref{reflection rep}} \label{symmetric group}
Theorem \ref{Lorenzini applied} (i) gives us the order of the abelian group $K(\gamma)$ and part (ii) gives some restrictions on the structure of this group, however this is not enough to uniquely specify $K(\gamma)$ in general.  Below we analyse the Smith normal form of $tI-\tilde{C}$ over $\Z[t]$ for the irreducible reflection representation of $\mf{S}_n$, which will allow us to determine the critical group in this case.  First we will briefly review basic facts about Smith normal form.

Let $R$ be a commutative ring with unit and $A \in R^{n \times n}$ be a matrix.   A matrix $S$ is called the \textit{Smith normal form} of $A$ if:
\begin{itemize}
\item There exist invertible matrices $P,Q \in R^{n \times n}$ such that $S=PAQ$.
\item $S$ is a diagonal matrix $S=\diag(s_1,...,s_n)$ with $s_i | s_{i+1}$ for $i=1,..., n-1$.
\end{itemize}

The following proposition lists several well-known facts about Smith normal form, see for example \cite[Section 2]{SNF}.

\begin{prop} \label{SNF facts} Let $A \in R^{n \times n}$.
\begin{itemize}
\item[(i)] Suppose $A$ has Smith normal form $S=\diag(s_1,...,s_n)$.  Then 
\[
\coker(A:R^n \to R^n) \cong \bigoplus_{i=1}^n R/(s_i)
\]
\item[(ii)] If $R$ is a PID then $A$ has a Smith normal form.
\item[(iii)] If $R$ is a domain, then the Smith normal form of $A$, if it exists, is unique up to multiplication of the $s_i$ by units in $R$.
\end{itemize}
\end{prop}

Let $Y$ denote Young's lattice, the set of all integer partitions ordered by inclusion of Young diagrams.  Let $\Z Y$ denote the free abelian group with basis $Y$, and let $Y_i$ denote the $i$-th rank of $Y$.  For each $i$, define linear maps called the \textit{up and down maps} $U_i: \Z Y_i \to \Z Y_{i+1}$ and $D_{i+1}: \Z Y_{i+1} \to \Z Y_i$ by
\begin{align*}
&U_i(y) = \sum_{\substack{z \in Y_{i+1} \\ y \leq z}} z 
&D_{i+1}(y) = \sum_{\substack{x \in Y_{i} \\ x \leq y}} x 
\end{align*} 
We will suppress the subscripts when they are clear from context.  Let $p(n)$ denote the number of integer partitions of $n$.

\begin{theorem}\cite[Theorem 1.2]{CS}
\label{stanley}
The matrix $U_{n-1}D_n+tI$ has Smith normal form \\ $\diag(\alpha_{p(n)}(t),...,\alpha_1(t))$ over $\Z[t]$ where
\[
\alpha_i(t)=\prod_{\substack{0 \leq k \leq n \\ p(n-k)-p(n-k-1)\geq i}}(t+k)
\]
\end{theorem}

The following proposition gives a character formula expressing the rows of $M$ associated to the reflection representation of $\mf{S}_n$ (Ballantine and Orellana use the equivalent language of Schur functions and Kronecker products). This will allow us to relate $\tilde{C}$ to $UD$ in this case.

\begin{prop}\cite[proof of Proposition 4.1]{BO}
\label{kronecker}
Let $n$ be a positive integer and $\lambda \vdash n$ then
\begin{equation}\label{UD-I}
\chi_{(n-1,1)} \cdot \chi_{\lambda} = C(\lambda)\chi_{\lambda} + \sum \chi_{\mu}
\end{equation}
where $C(\lambda)=|\{i|\lambda_i > \lambda_{i+1}, 1 \leq i \leq l(\lambda)-1\}|$ and the sum is over all partitions different from $\lambda$ that can be obtained by removing and then adding a corner cell to $\lambda$.
\end{prop}

We can now prove Theorem \ref{reflection rep}.

\begin{thm5}
Let $\gamma$ be the reflection representation of $\mf{S}_n$ and let $p(j)$ denote the number of partitions of the integer $j$. Then
\[
K(\gamma) \cong \bigoplus_{i=2}^{p(n)-p(n-1)} \Z/q_i \Z
\]
where
\[
q_i=\prod_{\substack{1 \leq j \leq n \\ p(j)-p(j-1)\geq i}}j
\]
\end{thm5}
\begin{proof}
The right hand side of Equation (\ref{UD-I}) can easily be seen to be the sum of the irreducible representations indexed by the partitions appearing in $(UD-I)\lambda$ since $|C(\lambda)|$ is one less than the number of corner cells in $\lambda$ (the set $C(\lambda)$ excludes the last corner).  Therefore $M=UD-I$ and $\tilde{C}=(n-1)I-M=nI-UD$.  Thus the formula in Theorem \ref{stanley} gives the Smith normal form over $\Z[t]$ of $(n+t)I-\tilde{C}$.  Replacing $t$ by $t-n$ shows that the Smith normal form of $tI- \tilde{C}$ is $\diag(\beta_{p(n)},...,\beta_1)$ with 
\[
\beta_i(t)=\prod_{\substack{0 \leq k \leq n \\ p(n-k)-p(n-k-1) \geq i}} (t-(n-k))
\]
Finally, reindexing with $j=n-k$ and setting $t = 0$ gives the desired result.
\end{proof}
\begin{remark}
The existence of Smith normal forms over $\Z[t]$ for the operators $UD+tI$ in the differential posets $Y^r$ for $r>1$ was recently proven by Nie \cite{Nie} (see Shah \cite{Shah} for another proof), generalizing Theorem \ref{stanley}.  Of particular interest in the context of critical groups is the case $r=2$, since the elements of $Y^2$ correspond to irreducible representations of the Weyl group of type $B_n$, just as $Y$ indexes the irreducibles of $\mf{S}_n$.  One might have hoped that $UD-cI$ would again have encoded $\gamma \otimes (\text{--})$, as in Proposition \ref{kronecker}, for some irreducible representation $\gamma$, thus allowing us to compute the critical group corresponding to this representation.  Some computations for small values of $n$, however, show that this is not the case.
\end{remark}

\section*{Acknowledgements}
This research was originally conducted for my undergraduate honors thesis at the University of Minnesota.  I am very thankful to Vic Reiner, my thesis advisor, for introducing me to the topic of critical groups and for always knowing the right references to check and the right ideas to try, as well as for his careful reading of earlier drafts of this paper.  I also thank Gregg Musiker and Pavlo Pylyavskyy for agreeing to be readers for my thesis.

\end{document}